\documentclass[hidelinks]{amsart}
\usepackage[english]{babel}
\usepackage[utf8]{inputenc}
\usepackage[T1]{fontenc}
\usepackage{lmodern}
\usepackage{latexsym}
\usepackage{amsthm}
\usepackage{a4}
\usepackage{amsfonts}
\usepackage{bm}
\usepackage{amssymb}
\usepackage{amsmath}
\usepackage{subfigure}
\usepackage{eucal}
\usepackage{amscd}
\usepackage{amsmath,mathtools}
\usepackage{array}
\usepackage{float}
\usepackage{setspace}
\usepackage{listings}
\usepackage{graphicx}
\usepackage{xifthen}

\usepackage[style=alphabetic,backend=biber,hyperref=true,giveninits=true]{biblatex}
\usepackage{subfiles}
\usepackage{hyperref}
\usepackage{csquotes}
\usepackage{wrapfig}
\usepackage{array}
\usepackage{lipsum}
\usepackage{verbatim}
\usepackage{nicematrix}

\usepackage{tikz-cd}
\usetikzlibrary{intersections}
\usetikzlibrary{math}

\tikzset{
	partial ellipse/.style args={#1:#2:#3}{
		insert path={+ (#1:#3) arc (#1:#2:#3)}
	}
}

\usepackage{tikz}
\usepackage{pgfplots}
\pgfplotsset{compat=newest}

\newtheorem{theorem}{Theorem}[section]
\newtheorem{proposition}[theorem]{Proposition}
\newtheorem{lemma}[theorem]{Lemma}

\theoremstyle{definition}
\newtheorem{definition}[theorem]{Definition}

\newtheorem{remark}[theorem]{Remark}

\newcommand{\Ho}{\textup{H}}

\newcommand{\Pic}{\operatorname{Pic}}

\newcommand{\Id}{\operatorname{Id}}
\newcommand{\Image}{\operatorname{Im}}
\newcommand{\Ker}{\operatorname{Ker}}
\newcommand{\Zt}{\mathbb{Z}/2\mathbb{Z}}

\newcommand{\Bun}{\mathfrak{Bun}}
\newcommand{\EM}{EM^{0,\star}_2(g,n,r)}
\newcommand{\HHplus}{H^+ \cdot H^+}

\DeclareFieldFormat[article]{title}{\mkbibemph{#1}}
\DeclareFieldFormat[incollection]{title}{\mkbibemph{#1}}
\AtEveryBibitem{
	\clearlist{language}
	\clearfield{isbn}
	\clearfield{issn}
	\clearfield{pages}
	\clearfield{url}
}
\DefineBibliographyStrings{english}{
  in = {},
}
\DeclareFieldFormat[article]{number}{\mkbibparens{#1}} 
\renewbibmacro*{volume+number+eid}{%
  \printfield{volume}%
  \printfield{number}%
  \setunit{\addspace}%
  \printfield{eid}}

\DeclareFieldFormat{edition}{%
 {#1~ed.}}

\newbibmacro*{edition+year+pages}{%
  \iffieldundef{edition}
    {}
    {\setunit{}
     \mkbibparens{\printfield{edition}%
       \iffieldundef{year}
         {}
         {\space\printfield{year}}}}%
  \iffieldundef{pagetotal}
    {}
    {\addcomma\space\printfield{pagetotal}}%
}

\newbibmacro*{bookseries}{%
  \iffieldundef{series}
    {}
    {\printfield{series}%
     \iffieldundef{number}
       {}
       {\addspace\printfield{number}}%
     \addperiod\space}}

\DeclareBibliographyDriver{book}{%
  \usebibmacro{bibindex}%
  \usebibmacro{begentry}%
  \usebibmacro{author/editor+others/translator+others}%
  \setunit{\labelnamepunct}\newblock
  \usebibmacro{title}%
  \newunit\newblock
  \usebibmacro{bookseries} 
  \printlist{publisher}%
  \newunit
  \usebibmacro{edition+year+pages} 
  \newunit\newblock
  \usebibmacro{finentry}%
}

\begin{document}

\title[Cohomology of the Moduli Stacks of Real Vector Bundles]{Cohomology of the Moduli Stacks of Real Vector Bundles on Type I Real Algebraic Curves}
\author{Luca Dal Molin \and Frank Neumann}
\address{Dipartimento di Matematica\\ Università degli Studi di Trento\\ Trento}
\email{luca.dalmolin@unitn.it}
\address{Dipartimento di Matematica "Felice Casorati"\\ Università degli Studi di Pavia\\ Pavia}
\email{frank.neumann@unipv.it}

\begin{abstract}
	We study the moduli stacks of real vector bundles of fixed rank and degree on a type I real algebraic curve and determine its mod $2$ cohomology algebra in terms of characteristic classes induced from the complex Atiyah-Bott classes.
\end{abstract}

\subjclass[2020]{\textbf{14H60}; 14D23, 14P25}
\keywords{cohomology of moduli stack, real algebraic curves, real vector bundles.}

\maketitle

\section{Introduction}

\noindent The scope of this paper is to present a detailed calculation of the mod $2$ cohomology algebra of the moduli stack of real vector bundles of fixed rank and degree over a real algebraic curve together with a complete and explicit description  of the cohomology classes that generate this algebra. Our primary focus here will be the geometrically significant case of real algebraic curves of type I, namely those for which the real structure disconnects the associated compact Riemann surface into two connected components and, equivalently, the quotient surface is orientable. Real algebraic curves of this type occupy a distinguished position in the theory of real algebraic curves, both because of their rich topological structure and because they provide a setting in which methods from gauge theory, equivariant algebraic topology, and the geometry of moduli spaces interact in a particular transparent manner.

The topology of moduli spaces of vector bundles over complex algebraic curves has been extensively studied since the foundational work of Atiyah and Bott \cite{zbMATH03803614}, who interpreted the Yang--Mills equations on a Riemann surface as an equivariant Morse-theoretic problem on the space of connections. Their approach led to a far-reaching description of the cohomology of moduli spaces and stacks of vector bundles in terms of natural universal characteristic classes. Subsequent works, notably those of Kirwan, Newstead, and others \cite{zbMATH05827280,zbMATH05831562,zbMATH05605320}, refined these methods and produced a detailed understanding of the cohomology algebra in the complex case. In particular, the moduli stack of vector bundles over a compact Riemann surface admits a rich and highly structured cohomology algebra generated by tautological classes arising from universal bundles and their characteristic classes. In contrast, the corresponding theory for real algebraic curves remains substantially less developed. The presence of an anti-holomorphic involution introduces new topological phenomena that are invisible in the complex setting. Real vector bundles also carry additional discrete invariants, and the associated moduli spaces inherit subtle equivariant structures that significantly complicate their topology. Nevertheless, important progress has been achieved in recent years through the works of Liu--Schaffhauser \cite{zbMATH06272210} and Baird \cite{zbMATH06355665}, who computed Poincaré series for moduli spaces of semistable real vector bundles. Although their methods differ substantially in character, both approaches reveal that the natural object underlying these calculations is not merely the coarse moduli space, but rather the full moduli stack itself as it retains the automorphism data and equivariant information necessary for a conceptual understanding of the topology.

The present work is motivated by the idea that the stack-theoretic perspective should provide the natural framework for extending the Atiyah--Bott picture to the real setting. Our starting point is given by the works of Liu--Schaffhauser \cite{zbMATH06272210} and Baird \cite{zbMATH06355665} which provide a calculations for the Poincaré series of the related coarse moduli space. But as noted by the authors, in order to obtain the desired result they have to pass through the Poincaré series of the whole moduli stack which is our genuine interest here. Our goal here is not only to recover numerical invariants such as Poincaré series, but also to obtain a structural description of the whole mod $2$ cohomology algebra itself. To this end, we develop a unified approach that synthesizes and stackifies the methods introduced by Liu--Schaffhauser and Baird. The guiding principle is that the equivariant and stratified structures appearing in the study of real gauge theory should admit a direct interpretation at the level of the moduli stack, thereby allowing one to identify canonical cohomology generators of the cohomology algebra.

Our adaptation is motivated by related calculations in the complex case. It traces back to the fundamental work of Atiyah-Bott \cite{zbMATH03803614}. In the complex case we have a complete description of the cohomology algebra of the  moduli stack of vector bundles of fixed rank and degree over a Riemann surface (see \cite{zbMATH05827280}, \cite{zbMATH05831562} and \cite{zbMATH05605320}). A central role in our analysis is played by an adaption of the complex Atiyah-Bott characteristic classes to the real setting. In addition to the classical characteristic classes familiar from complex geometry, one encounters cohomological invariants associated with orthogonal and equivariant structures, among which the Dickson invariants appear naturally. These classes encode subtle information about the topology of real bundles and their symmetry groups, and they provide the algebraic framework necessary for describing the mod $2$ cohomology algebra explicitly. One of the main objectives of this article is to explain how these invariants arise geometrically and how they generate the cohomology algebra of the moduli stack. The restriction to type I curves allows several important simplifications while still retaining the essential geometric features. In this case the decomposition of the Riemann surface induced by the real structure leads to a particularly tractable description of the relevant gauge-theoretic moduli spaces and their equivariant cohomology. Moreover, the topology of the real locus interacts favorably with the Harder--Narasimhan stratification, making it possible to carry out explicit calculations and identify the corresponding universal classes in a concrete way. In follow-up work we will address the cases of type 0 and type II curves.

\noindent This article is organized as follows: In Sections 2 and 3 we recall the necessary background on real algebraic curves, real vector bundles, and the relevant cohomological invariants. Particular emphasis is placed on the topology of type I curves, equivariant cohomology techniques, and the role of Dickson invariants in the description of the cohomology algebra. Section \ref{sec:rk1type1} treats the rank $1$ case in detail, where the essential geometric ideas already become visible. We then extend these constructions to arbitrary rank real vector bundles and derive the general description of the mod $2$ cohomology generators.

\section{Real algebraic curves}

Let $M$ be a connected Riemann surface, from which we will construct our real algebraic curve. In order to do so, we will follow closely Atiyah's approach in \cite{zbMATH03235829} on real vector bundles over real algebraic curves.

\begin{definition}
	We call a pair $(M,\sigma)$, where $M$ is a connected Riemann surface and $\sigma$ is an antiholomorphic involution, a \textit{Klein surface}. We will define a \textit{real algebraic curve} as the fixed point locus of the action of this involution on the Riemann surface, which we will denote as $M^\sigma$.
\end{definition}

A Klein surface $(M,\sigma)$ is completely characterized by the following topological invariants, as introduced in \cite{zbMATH02681780}:
\begin{itemize}
	\item $g$: the genus of the Riemann surface $M$, which we will always assume to be greater or equal $2$;
	\item $n$: the number of connected components of the fixed point locus $M^\sigma$;
	\item $a$: a value that can be $0$ or $1$ with respect to the number of connected components of the complement $M \setminus M^\sigma$ and hence the orientability of the quotient.
\end{itemize}

Using these invariants, especially the third one, we can divide real algebraic curves into three disjoint classes. These are the ones given by the following definition.

\begin{definition}
	Let $(M,\sigma)$ be a Klein surface. A real algebraic curve will be called:
\begin{itemize}
	\item \textit{type 0}: if $n=0$, so the involution $\sigma$ acts freely;\\
	\item \textit{type I}: if $a=0$, so $M \setminus M^\sigma$ is made of two distinct connected components;\\
	\item \textit{type II}: if $a=1$ and $n \neq 0$, so $M \setminus M^\sigma$ is made of a single connected component.
\end{itemize}
\end{definition}

\begin{proposition}
	Given a type 0 real algebraic curve $(M,\sigma)$, we have that $a=1$. Also the involution $\sigma$ is given by the antipodal map.
\end{proposition}

\begin{proof}
	This map is described explicitly in \cite{zbMATH07144072}. In particular, having $\sigma$ acting freely, we obtain that $M^\sigma=\varnothing$ and so $M \setminus M^\sigma$ is made of a single connected component and therefore $a=1$.
\end{proof}

\begin{remark}
	Usually the type 0 case is seen as a sub case of type II real algebraic curves. Here we consider them as disjoint cases, in order to have a more clear exposition. This is also the reason for why we do not need to put the condition $n \neq 0$ for type I real algebraic curves. Having that $M \setminus M^\sigma$ splits into two distinct connected components already implies this condition.
\end{remark}

Between the three invariants $g, n, a$ we have a strong relation. In particular, there exists a number $g^\prime \geq 0$ such that:
\begin{equation}
	\label{eq:gprime}
	g=2g^\prime+n-1, \text{\phantom{aaaaaaa} if } a=0
\end{equation}
\begin{equation}
	\label{eq:gprime0}
	g=2g^\prime+n+c, \text{\phantom{aaaaaaa} if } a=1
\end{equation}
where $c$ is equal to 1 if $g-n$ is odd or 0 if $g-n$ is even.

Concluding we state Harnack's inequality (originally introduced in \cite{zbMATH02714346}) which gives us the following relation:
\[
	n \leq g+1.
\]
This motivates the following definition (see \cite{zbMATH04210331}):

\begin{definition}
	\label{def:MCurve}
	Let $(M,\sigma)$ be a Klein surface, then if $g=n+1$ we call the associated real algebraic curve an \textit{M--curve} or \textit{maximal curve}.
\end{definition}

\begin{proposition}
	If $(M,\sigma)$ is an M--curve, then it is of type I.
\end{proposition}

\begin{proof}
	This is a classical results (see )\cite{zbMATH04210331}). It can be also derived from equations (\ref{eq:gprime}) and (\ref{eq:gprime0}). Just fix the invariants $g$ and $n=g+1$, now it must exists a $g^\prime$ such that one of the two equations is verified, and for equation \ref{eq:gprime0} we get $c=1$, because $g-n=1$.

	Now if we suppose $a=1$ we get
		$-2 = 2g^\prime,$
	which is impossible, while equation (\ref{eq:gprime}) is realized with $g^\prime=0$ and so we have $a=0$.
\end{proof}

Now we can look at vector bundles defined over a Klein  surface and start to define the moduli stacks.

\begin{definition}
	A smooth complex vector bundle $E$ over $(M,\sigma)$ is called a \textit{real vector bundle} when it is equipped with a function $\tau$ such that:
	\begin{itemize}
	\item $\tau$ is $\mathbb{C}$--antilinear on the fiber,
	\item $\tau^2=\Id_E$,
	\item the following diagram commutes:
	\[
		\begin{tikzcd}[ampersand replacement=\&]
			E \arrow[r,"\tau"]\arrow[d] \& E \arrow[d] \\
			M \arrow[r,"\sigma"] \& M
		\end{tikzcd}
	\]
	\end{itemize}
\end{definition}

We also have to define the concept of gauge transformation that will be useful in various discussions.

\begin{definition}
	\label{def:gauge}
	Let $E \rightarrow M$ be an holomorphic vector bundle. The \textit{complex gauge group} $\mathcal{G}_\mathbb{C}$ is given by all maps $f: E \rightarrow E$ such that the following diagram commutes:
	\[
		\begin{tikzcd}[ampersand replacement=\&]
			E \arrow[r,"f"]\arrow[d] \& E \arrow[d] \\
			M \arrow[r,"\Id"] \& M
		\end{tikzcd}
	\]
	We will also say that $f$ \textit{lifts} the identity map.
\end{definition}

With these definitions, we can now construct the moduli stack of real vector bundles over a real algebraic curve in the following way (see \cite{zbMATH06272210}):
\begin{definition}
	\label{def:realstack}
	Let $(E,\tau)$ be a real vector bundle of rank $r$ and degree $d$ over a fixed Klein surface $(M,\sigma)$ and consider the space $\mathcal{C}^\tau$ of all possible holomorphic structures compatible with $\tau$ along with the subgroup of the complex gauge group $\mathcal{G}_\mathbb{C}^\tau \subset \mathcal{G}_\mathbb{C}$ given by all elements that commute with $\tau$. The associated moduli stack is defined as the quotient stack
	\[
		\Bun^{r,d,\tau}_{M,\sigma}:=[\mathcal{C}^\tau/\mathcal{G}_\mathbb{C}^\tau].
	\]
	It classifies the rank $r$ and degree $d$ vector bundles with real structure $\tau$ over the Klein surface $(M,\sigma)$.
\end{definition}

A fundamental property described in \cite{zbMATH06141803} is that for all possible choices of a real structure $\tau$ we have that the moduli stack $\Bun^{r,d,\tau}_{M,\sigma}$ in fact describes a connected component of the fixed point locus of the complex moduli stack and that these components are all homeomorphic. But these do not give all components of the complex moduli stack, because there will also be quaternionic ones, namely for the case that $\tau^2=-\Id$. We will not take these quaternionic components into account here though.

But also the converse is true, namely for any connected component of the complex moduli stack that consists of real vector bundles there exists a real structure $\tau$ such that the considered component can be identified with the real moduli stack $\Bun^{r,d,\tau}_{M,\sigma}$.

The real moduli stack is built in the same way as the complex moduli stack $\Bun^{r,d}_M$ of complex vector bundles of rank $r$ and degree $d$ over the Riemann surface $M$
\[
	\Bun^{r,d}_M=[\mathcal{C}/\mathcal{G}_\mathbb{C}],
\]
where $\mathcal{C}$ is the set of holomorphic structures and $\mathcal{G}_\mathbb{C}$ is the complex gauge group (\cite[Section 1.2 and 1.3]{zbMATH06272210}). Having that $\mathcal{C}^\tau$ is contractible, the real moduli stack has in fact the homotopy type of the classifying space $B\mathcal{G}_\mathbb{C}^\tau$ of the real gauge group $\mathcal{G}_\mathbb{C}^\tau$.

\begin{proposition}
	Let $(M,\sigma)$ be a Klein surface. For each choice of $r,d,\tau$ the moduli stack $\Bun^{r,d,\tau}_{M,\sigma}$ is an algebraic stack.
\end{proposition}

\begin{proof}
	As we have seen above, the moduli stack $\Bun^{r,d,\tau}_{M,\sigma}$ is a connected component of the fixed point locus of the moduli stack $\Bun^{r,d}_M$ by the map induced from the action of $\sigma$ on vector bundles $\mathcal{E}$, namely:
		\begin{tikzcd}[ampersand replacement=\&]
			\mathcal{E} \arrow[r] \& \overline{\sigma^\star\mathcal{E}}.
		\end{tikzcd}
	Now having an action of the group $\Zt$ generated by $\sigma$ on $\Bun^{r,d}_M$, we can apply \cite[Theorem 3.3]{zbMATH02187053}. From this we get that the fixed point stack is still algebraic, and so are its connected components.
\end{proof}

\section{Construction of a cohomology basis of a Klein surface and the Dickson invariant}
\label{sec:basisConstruction}

Let, as above, $(M,\sigma)$ be a Klein surface, where $M$ is of genus $g\geq 2$. The involution $\sigma$ of $M$ induces an involution $\sigma^\star$ on the cohomology $\Ho^1(M,\Zt)$. Remembering that the dimension of this first cohomology group is $2g$, once we choose a basis, we can represent $\sigma^\star$ by a $2g \times 2g$--matrix over $\Zt$.

From the article of Levine and Nahikian \cite{zbMATH03183062} we know that there exists a choice of a basis such that the matrix of $\sigma^\star$ is of the following type:
\begin{equation}
	\label{eq:matrixSigma}
	\begin{pNiceArray}{c c}
		\Id_{2g-2s} & \mathbf{0} \\
		\mathbf{0} & K_{2s} \\
	\end{pNiceArray},
\end{equation}
where $0 < s \leq g$ and $K_{2s}$ is the direct sum of $s$ matrices of the form:
\[
	\begin{pNiceArray}{c c}
        \mathbf{1} & \mathbf{1} \\
        \mathbf{0} & \mathbf{1} \\
    \end{pNiceArray}.
\]

Let us recall now briefly the explicit description of the generators for the cohomology of the complex moduli stack of vector bundles over a Riemann surface (see details for example in \cite[Ch. 3]{zbMATH05605320}, \cite{zbMATH05831562} or \cite{zbMATH05827280}). We first fix a basis $(\alpha_j)_{j=1,\dots,2g}$ of $\Ho^1(M,\Zt)$. For the moduli stack $\Bun^{r,d}_M$ of rank $r$ and degree $d$ vector bundles over $M$, we consider the universal vector bundle over $\Bun^{r,d}_M\times M$ which gives a morphism of algebraic stacks
\[
\begin{tikzcd}[ampersand replacement=\&]
	\Bun^{r,d}_M \times M \arrow[r] \& BGL_r,
\end{tikzcd}
\]
where $BGL_r$ is the classifying stack of all rank $r$ vector bundles. We can consider the pullbacks of the Chern classes of the universal bundle over $BGL_r$ with respect to this morphism. Now thanks to Künneth decomposition we get the following description for the Chern classes of the universal bundle over $\Bun^{r,d}_M\times M$ for each $i \leq r$ and $j=1, 2, \ldots, 2g$
\begin{equation}
	\label{eq:neuGenerator}
	c_i \otimes 1 + \sum^{2g}_{j=1} a^{(i)}_j \otimes \alpha_j + f_i \otimes \left[M\right].
\end{equation}

This gives us the Atiyah-Bott classes $c_i, f_i, a^{(i)}_j$ for $i=1, 2, \ldots r$ and $j=1, 2,\ldots, 2g$ as generators for the cohomology algebra $\Ho^\star(\Bun^{r,d}_M,\Zt)$ of the complex moduli stack $\Bun^{r,d}_M$.

In order to obtain the generators for the real moduli stack, we need to describe the matrix given in equation (\ref{eq:matrixSigma}) in details, and also give an explicit description of the basis of $\Ho^1(M,\Zt)$ chosen in such a way that the matrix is directly expressed in terms of $\sigma$ and its action on $M$. We will also need a more direct interpretation of the Atiyah-Bott cohomology classes that generate the cohomology of the moduli stack of complex vector bundles in our specific setting. We will discuss all of these points in the next subsections.

\subsection{The Dickson invariant}
First of all we introduce the Dickson invariant of the involution $\sigma$. This invariant permits us to distinguish the conjugacy classes of the involution. It will be equal to the number $s$ which appeared already in equation (\ref{eq:matrixSigma}). This follows from the discussion in \cite{arXiv:1612.08487, zbMATH07144072}.

\begin{definition}
	Let $V$ be a symplectic vector space of dimension $2g$ over $\Zt$. Given an involution $\sigma$ over $V$, the \textit{Dickson invariant} of $\sigma$ is an integer $0 \leq D \leq g$ given as the rank of the map
	$\sigma + \Id$.
\end{definition}

From \cite[Chapter 5]{zbMATH07144072} we get a full description of all the possible involutions that we can have over a surface of genus $g$.

In particular all the involutions that give us type I curves are the ones of the reflection form (see figure \ref{fig:reflection}), while the ones that give type II curves are given by some kind of antipodal form (see figure \ref{fig:ddOld}). By this we get that we can classify both classes by the number of holes that they reflect, that is equal to the $g^\prime$ of equation (\ref{eq:gprime}) or (\ref{eq:gprime0}).

\begin{proposition}
	\label{prop:dick}
	For a type I curve $(M,\sigma)$ we have that:
	\[
		D = 2g^\prime.
	\]
\end{proposition}

\begin{proof}
	First of all we notice that if $\alpha \in \Image(\sigma^\star + \Id)$ then $\sigma^\star(\alpha)=\alpha$ because $\sigma^\star$ is an involution. Also if $\sigma^\star(\alpha)=\alpha$ then $(\sigma^\star + \Id)(\alpha)=0$.

	If we take $M/\sigma$ as the quotient space we know that it is isomorphic to $\Sigma_{g^\prime,n}$, where $\Sigma_{g^\prime,n}$ is a Riemann surface of genus $g^\prime$, defined in equation (\ref{eq:gprime}), with $n$ open disks removed.

	Let $\pi$ be the projection map and let $\pi^\star$ be the induced map for the first cohomology group. Obviously we have that the elements of $\Image(\pi^\star)$ are $\sigma^\star$--invariant by the definition of the quotient map.

	Now in $\Ho^1(M/\sigma,\Zt)$ we can identify two types of classes, precisely $2g^\prime$ classes given by the generators of $\Ho^1(T_{g^\prime},\Zt)$, where $T_g^\prime$ is a Riemann surface of genus $g^\prime$, and $n-1$ classes added by removing the $n$ disks (the last one is a combination of these). This set of classes gives a basis for the first cohomology $\Ho^1(M/\sigma,\Zt)$.

	But now the $n-1$ classes added by the disks live inside the reflection plane and so cannot be in the image of $\sigma^\star + \Id$, the same is true for the other $n-1$ classes that do not live in $\Image(\pi^\star)$ and that are cut in half by the plane. This observation leaves us exactly $2g^\prime$ classes that are not $\sigma^\star$--invariants that, under the action of $\sigma^\star + \Id$, gives us the other $2g^\prime$ classes that generate $\Image(\sigma^\star + \Id)$ and so for the Dickson invariant we get $D=2g^\prime$.
\end{proof}

\begin{proposition}
	\label{prop:dickII}
	For a type II curve $(M,\sigma)$ we have that:
	\[
		D=2g^\prime+c,
	\]
	where $c=1$ if $g-n$ is even or $c=2$ if $g-n$ is odd.
\end{proposition}

\begin{proof}
	We note that in \cite[Proposition 9.13]{zbMATH07144072} a less general version of this proposition is proved. In particular, the two cases with $g^\prime$ being equal to $0$ are considered, and it is shown that $D$ is equal to $0$ or $1$ as expected. Furthermore only the special case $n=3$ is considered, but this limitation can be removed.

	First of all we fix the invariants of the curve, namely the genus $g$ and the number of connected components $n$. We will give the proof for $n=3$, the generalization for arbitrary $n$ is then straightforward. It is simple to see that in this case the parity of $g-n$ is the opposite of the one of $g$.

	Similar as in \cite{zbMATH07144072} we can model the curve in the following way:
	\begin{itemize}
		\item First of all take a surface $T_{g-n}$ of genus $g-n$ with antipodal involution acting on it;
		\item then take an open cylinder with the mid plane reflection and glue it to $T_{g-n}$ by removing two antipodal disks from it;
		\item repeat this procedure with $n$ cylinders, removing each time the disks from the previous cylinder.
	\end{itemize}

	We can also understand this as a double connected sum of a surface $T^\prime_{n-1}$ of genus $n-1$ with $T_{g-n}$ where the involution is given by the reflection through a plane on $T^\prime_{n-1}$ and the antipodal map on $T_{g-n}$.

	We illustrate this procedure in the following figures for the case of $n=3$:

	\begin{figure}[H]
		\centering
    	\begin{minipage}{0.45\textwidth}
        	\centering
        	\includegraphics[width=\textwidth]{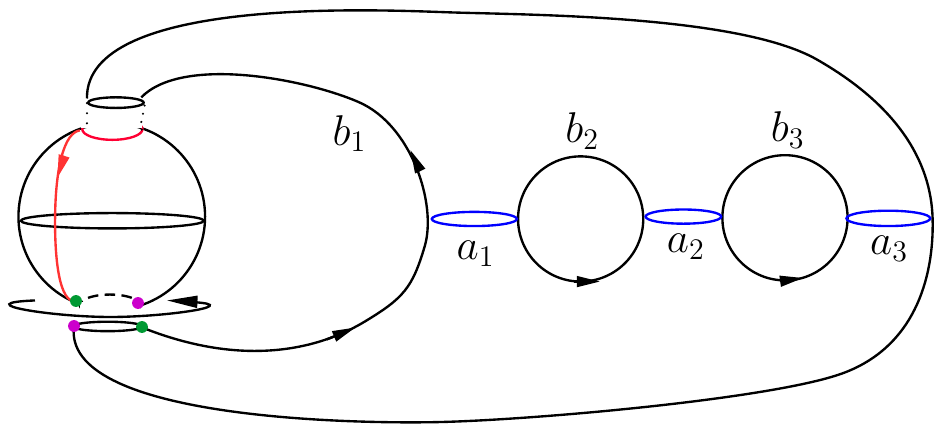}
			\caption{\\Case $g=n$}
			\label{fig:ddOld}
    	\end{minipage}\hfill
    	\begin{minipage}{0.45\textwidth}
        	\centering
        	\includegraphics[width=\textwidth]{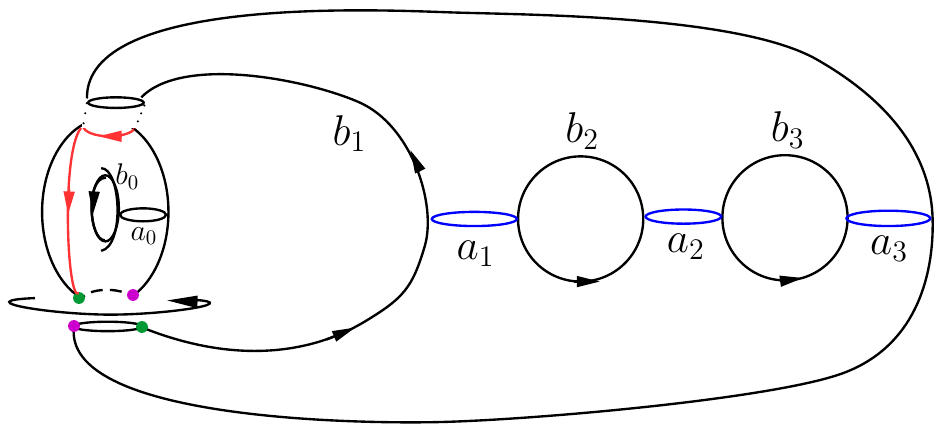}
			\caption{\\Case $g=n+1$}
    	\end{minipage}
	\end{figure}

	More generally we have the following picture:

	\begin{center}
		\begin{figure}[H]
			\includegraphics[width=0.75\textwidth]{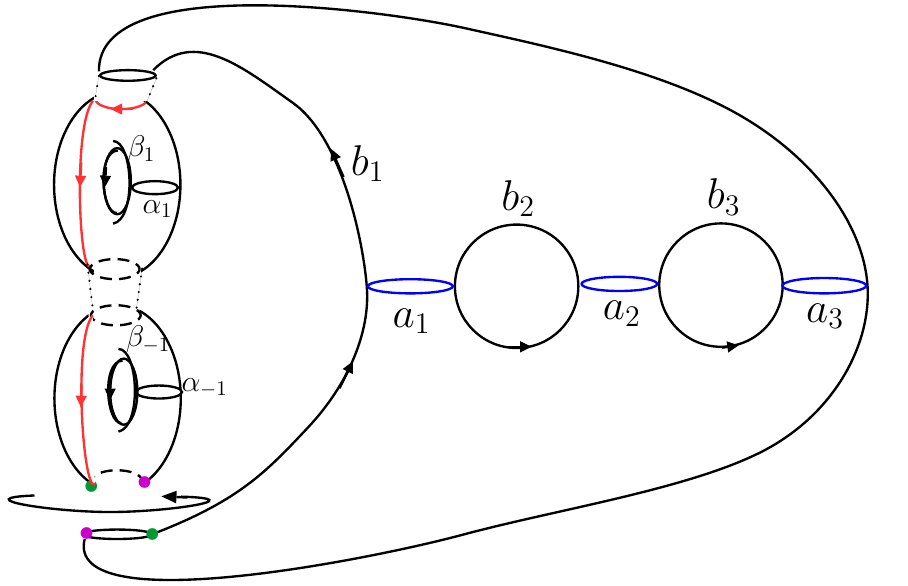}
			\caption{General case}
			\label{fig:ddNew}
		\end{figure}
	\end{center}

	Now, as illustrated in figure \ref{fig:ddNew}, we call $a_1,a_2,a_3,b_1,b_2,b_3$ the basis of $T_3$ once we remove two disks and glue with $T_{g-3}$ (we can imagine it as a surface of genus $3+1$), while the generators of $T_{g-3}$ will be called, from top to bottom, $\alpha_1=\alpha_{-(g-n)},\dots,\alpha_{g-n}=\alpha_{-1}$ and $\beta_1=\beta_{-(g-n)},\dots,\beta_{g-n}=\beta_{-1}$. We have that $\sigma^\star$ keeps the $a_i$ fixed for each $i$ and $b_i$ for $i \neq 1$, while it sends $\alpha_j$ into $\alpha_{-j}$ and similarly for the $\beta_j$. There are two cases to be distinguished:

	\textit{Case $g-n$ even}:\\
	If we take $\alpha_i + \sigma^\star(\alpha_i{-i})$, it is cohomologous to the circuit $c$ given by $a_1+a_2+a_3$. In fact $\sigma(\alpha_{-i})$ gives us the circuit parallel to $\alpha_i$, and so we have that $\sigma^\star(\alpha_{-i})=\alpha_i+c$ for each $i$. And then $(\sigma+\Id)^\star(\alpha_i)=\alpha_i+\alpha_{-i}+c$.

	In the same way it is easy to see that $\beta_i+\beta_{-i} \neq 0$ can be obtained from $(\sigma+\Id)^\star(\beta_i)$.

	Finally we are left with $\sigma^\star(b_1)$, but in the same fashion as \cite{zbMATH07144072} we can show that $\sigma(b_1)+b_1$ is cohomologous to $c+\sum^{g-n}_{i=1}\beta_i$.

	From this we conclude that $D=2g^\prime+1$ and the generators are given by $\alpha_i+\alpha_{-i}$, $\beta_i+\beta_{-i}$ and $c$. The first two sets of generators give exactly $2g^\prime$ generators by construction.

	\textit{Case $g-n$ odd}:\\
	The process is the same as for the case $g-n$ even. The only difference is that now we have an index, $k=\lceil \frac{g-n}{2} \rceil$, such that $\alpha_k=\alpha_{-k}$ and $\beta_k=\beta_{-k}$.

	Now as before we have that $(\sigma+\Id)^\star(\alpha_i)=\alpha_i+\alpha_{-i}+c=c$ and so we get $c$ as a generator. Furthermore $\beta_k$ is fixed by $\sigma^\star$, but if we consider $(\sigma+\Id)^\star(b_1)$ we obtain, as above, that $\sigma(b_1)+b_1$ is cohomologous to $c+\beta_k+\sum^{g-n-1}_{i=1}\beta_i$.

	From this we conclude that $D=2g^\prime+2$ and the generators are given by $\alpha_i+\alpha_{-i}$, $\beta_i+\beta_{-i}$, $\beta_k$ and $c$. Again the first two sets of generators give exactly $2g^\prime$ generators.
\end{proof}

\subsection{Completion of the cohomology basis}
\label{sub:basis}
First of all we fix a basis of the subspace $\Image(\sigma^\star+\Id)$ and call it $(\alpha_i)_{i=1,\dots,D}$. By definition it is made up of $D$ elements.

Now consider the projection map $M \rightarrow M/\sigma$ to the quotient space of the action $\sigma$. WE have the following induced map in cohomology:
\[
	\begin{tikzcd}[ampersand replacement=\&]
		\pi^\star: \Ho^1(M/\sigma,\Zt) \arrow[r] \& \Ho^1(M,\Zt).
	\end{tikzcd}
\]

Because we have that $\sigma^\star(\alpha_i)=\alpha_i$ there exists classes $(\alpha^\prime_i)_{i=1,\dots,D}$ in $\Ho^1(M/\sigma,\Zt)$ such that $\pi^\star(\alpha^\prime_i)=\alpha_i$. We complete them to a basis $(\alpha^\prime_i)_{i=1,\dots,D},(\beta^\prime_i)_{i=1,\dots,n-1}$ of $\Ho^1(M/\sigma,\Zt)$. This follows because we have $\dim\left(\Ho^1(M/\sigma,\Zt)\right)=g$ and $g-D=n-1$.

Taking now the images through $\pi^\star$ of the above classes, and remembering that \[(\sigma^\star+\Id)^2=0,\] we can complete them to a basis $(\alpha_i)_{i=1,\dots,D},(\beta_i)_{i=1,\dots,n-1}(\gamma_i)_{i=1,\dots,n-1}$ of $\Ker(\sigma^\star+\Id)$. This follows because the elements $\beta_i$ are $\sigma^\star$ invariant and so they live in the kernel. Also by inspecting the matrix (\ref{eq:matrixSigma}) we get that $\dim(\Ker)=D+2n-2$.

Finally we obtain a basis for the cohomology $\Ho^1(M,\Zt)$ by simply taking:
\[
	(\alpha_i)_{i=1,\dots,D},(\beta_i)_{i=1,\dots,n-1}(\gamma_i)_{i=1,\dots,n-1},(A_i)_{i=1,\dots,D},
\]
with $(\sigma^\star+\Id)A_i=\alpha_i$. These last classes will be linearly independent because the $\alpha_i$ are linearly independent and they do not live in the kernel of the map by construction.

For completeness, to obtain exactly what we have written in equation (\ref{eq:matrixSigma}) as matrix of $\sigma^\star$, we have to reorder this basis in the following way:
\[
	(\beta_i, \gamma_i)_{i=1,\dots,n-1},(A_i,\alpha_i)_{i=1,\dots,D}.
\]
In fact we have that $\sigma^\star(A_i)=A_i+\alpha_i$ by construction, and the matrix of this transformation is therefore given as:
\[
	\begin{pNiceArray}{c c}
        \mathbf{1} & \mathbf{1} \\
        \mathbf{0} & \mathbf{1} \\
    \end{pNiceArray}.
\]

\section{Rank 1 case for type I curves}
\label{sec:rk1type1}

\subsection{Gerbe construction}

As a first task we want to build the moduli stack as part of an \mbox{$\mathbb{R}^\star$--gerbe} in the rank one case. In this way we will simplify the cohomology calculations that will follow.

Let us recall that, fixing a non--empty Klein Surface $(M,\sigma)$, we have that $\Pic^d(M)^\sigma$ is given by the disjoint union of $2^{n-1}$ connected components all diffeomorphic to $(S^1)^g$ inside the torus $\Pic^d(M)$, as can be seen from \cite{zbMATH03845739} and \cite{zbMATH06141803}.

From \cite{zbMATH03845739} and \cite{zbMATH06141803}, we also get that these connected components are indexed by the real structures $\tau$ that define a real bundle.
Also remembering that $\Pic^d(M)$ is a coarse moduli space and all the involutions are applied fiberwise, we obtain a map
\[
	\begin{tikzcd}[ampersand replacement=\&]
		\Bun^{1,d,\tau}_{M,\sigma} \arrow[r] \& \Pic^d(M)^\sigma_\tau,
	\end{tikzcd}
\]
such that $\Pic^d(M)^\sigma_\tau$ is a coarse moduli space for the moduli stack in the rank one case.

Now remember that, by definition, this map produces isomorphism classes. This is coherent with the construction because we have chosen $r=1$. In this case we also have that all the bundles that we are considering are actually stable bundles. Furthermore we also know that the geometric fiber of this map is given by $B\mathbb{R}^\star$, the classifying stack of $\mathbb{R}^\star$. This follows because a real stable bundle has only $\mathbb{R}^\star$ as real automorphism group.

We now have the following proposition:

\begin{proposition}
	\label{prop:split}
	The moduli stack $\Bun^{1,d,\tau}_{M,\sigma}$ splits in the following way:
	\[
		\Bun^{1,d,\tau}_{M,\sigma} \cong B\mathbb{R}^\star \times \Pic^d(M)^\sigma_\tau.
	\]
\end{proposition}

\begin{proof}
	From \cite[Corollary 3.12]{zbMATH05831562} it follows that there exists a Poincaré family over $\Pic^d(M)$, so there exists a bundle $\mathcal{E}$ over $M \times \Pic^d(M)$ such that for each line bundle of degree $d$ $L$ we have that $\mathcal{E}_{|L} \in [L]$.

	Having $M$ of genus $g\geq 2$ and of type I, it follows that $\Pic^d(M)^\sigma$ is not empty. If $L$ is a real bundle of degree $d$ we have the following isomorphisms:
	\[
		\mathcal{E}_{|L} \cong L \cong \overline{\sigma^\star L} \cong \overline{\sigma^\star \mathcal{E}_{|L}},
	\]
	and so also $\mathcal{E}_{|L}$ is a real bundle.

	Now from \cite[Proposition 2.8]{zbMATH06141803} these two bundles are isomorphic also as real bundles. So we get that they live in the same connected component of $\Pic^d(M)^\sigma$.

	Therefore we get a Poincaré family over $\Pic^d(M)^\sigma_\tau$, and similar as in \cite[Lemma 3.10]{zbMATH05831562} we get that the gerbe splits as described.
\end{proof}

Applying now the Künneth decomposition in cohomology to the splitting of Proposition \ref{prop:split}, we obtain the following structure for the cohomology algebra of the moduli stack

\vspace*{-0.5cm}

\begin{equation}
	\label{eq:kunnethSplit}
	\Ho^\star(\Bun^{1,d,\tau}_{M,\sigma},\Zt)=\Zt[\omega_1] \otimes \Lambda[a_i]_{i=1,\dots,g},
\end{equation}

where every cohomology class is of degree $1$.

\subsection{Main theorem}

Before formulating and proving our main theorem in the rank 1 case for type I curves, we will need to recall some constructions, following \cite{zbMATH06355665}:

\begin{definition}
	The group $LU_r$ of continuous maps from $S^1$ to $U_r$  (the group of complex matrix of rank $r$ with determinant equal to $1$) is called the \textit{loop group}. With $L_0U_r$ we will identify the connected component of the identity.
\end{definition}

\begin{definition}
	\label{def:loop}
	The group $LU^\tau_r$ of gauge transformation (recall definition \ref{def:gauge}) of the real vector bundle of rank $r$, $(E,\tau) \rightarrow (S^1,\sigma)$, where $\sigma$ is the identity map of $S^1$ is called the \textit{real loop group}.
\end{definition}

\begin{proposition}
	\label{prop:loop}
	Let $BLU^{\Id}_r$ be the classifying space of the loop group $LU^{\Id}_r$. We have the following isomorphism:
	\[
		\Ho^\star(BLU^{\Id}_r,\Zt) \cong \Ho^\star(SO(r),\Zt) \otimes \Zt[\omega_1,\dots,\omega_r], \text{ where } \deg(\omega_k)=k.
	\]
\end{proposition}

\begin{proof}
	See \cite[Proposition 5.5]{zbMATH06355665}, where the $\omega_1, \dots, \omega_r$ are the universal Stiefel--Whitney classes.
\end{proof}

Now we can work out the generators for the cohomology algebra of the moduli stack of line bundles over a fixed type I curve $(M,\sigma)$, which we state as follows:

\begin{theorem}
	\label{th:rk1theorem}
	Given a type I curve $(M,\sigma)$ of genus $g \geq 2$ with $n$ connected components, the moduli stack $\Bun^{1,d,\tau}_{M,\sigma}$ has the following cohomology algebra:
	\[
		\Ho^\star(\Bun^{1,d,\tau}_{M,\sigma},\Zt)=\Zt[\omega_1] \otimes \Lambda[\alpha_i]_{i=1,\dots,2g^\prime} \otimes \Lambda[\beta_i]_{i=1,\dots,n-1},
	\]
	where all the cohomology classes are of degree $1$, $w_1$ is the universal Stiefel--Whitney class and the others are given by the images of the basis of $\Ho^1(M,\Zt)$ as constructed in Section \ref{sec:basisConstruction}.
\end{theorem}

\begin{proof}
	As a first step we look at the M--curves, so the type I curves such that $n=g+1$ (see Definition \ref{def:MCurve}). From \cite[Sec. 4.3, eq (4.2)]{zbMATH06272210} we get the following fibration of topological spaces
	\begin{equation}
	\label{eq:LSfiber}
		\begin{tikzcd}[ampersand replacement=\&]
			B(\Omega^2(U(r))) \arrow[r] \& B(\mathcal{G}_\mathbb{C}^\tau) \arrow[r] \& E(G^\tau_{(n,a)}(r)) \times_{G^\tau_{(n,a)}(r)} W^\tau_{(g,n,a)}(r,d),
		\end{tikzcd}
	\end{equation}
	where, for $\tau$ a real structure, we have
	\begin{equation}
		\label{eq:Gtau}
		G^\tau_{(n,a)}=U(r)^a \times (U(r)^\tau)^n, \text{\phantom{ab}} W^\tau_{(g,n,a)}(r,d)=U(r)^{g+a} \times \prod^n_{i=0} O(r)_{(-1)^{w^{(i)}}}.
	\end{equation}

	Here with $w^{(i)}$ we denote the restriction to the $i$--th connected component of the fixed points set $M^\sigma$ of the first universal Stiefel--Whitney class. We recall that $M^\sigma$ is given by a disjoint union of $S^1$, while the fixed point of the real vector bundle $E^\tau$ is a collection of vector bundles over $\mathbb{R}$ over each $S^1$. Once we restrict over a single $S^1_{(i)}$ we get $w^{(i)}=\pm 1$, respecting the orientation of $E_{|S^1_{(i)}}$. Now if we take the orthogonal group of rank $r$, $O(r)$, we can split it in the special orthogonal group $SO(r)$ of matrix with determinant $+1$ and in its complement of matrix with determinant $-1$. In equation (\ref{eq:Gtau}) we take as $O(r)_{(-1)^{w^{(i)}}}$ the connected component of $O(r)$ of matrix with determinant equal to $w^{(i)}$.

	Now we can reduce this to our case with $a=0$, $n=g+1$ and $r=1$. This permits to simplify the calculations, in fact the fibration now becomes:
	\[
		\begin{tikzcd}[ampersand replacement=\&]
			\{\ast\} \arrow[r] \& B(\mathcal{G}_\mathbb{C}^\tau) \arrow[r] \& E(G^\tau_{(g+1,0)}(1)) \times_{G^\tau_{(g+1,0)}(1)} W^\tau_{(g,g+1,0)}(1,d),
		\end{tikzcd}
	\]
	and so we have
	\[B(\mathcal{G}_\mathbb{C}^\tau)\simeq E(G^\tau_{(g+1,0)}(1)) \times_{G^\tau_{(g+1,0)}(1)} W^\tau_{(g,g+1,0)}(1,d).
	\]
	Furthermore we can identify
	\[
		G^\tau_{(g+1,0)}(1)=(U(1)^\tau)^{g+1}=(\{\pm 1\})^{g+1}, \text{\phantom{ab}} W^\tau_{(g,g+1,0)}(1,d)=U(1)^g.
	\]
    Finally it remains only to calculate the equivariant cohomology algebra, but this was already done in \cite[Sec. 4.2.5]{zbMATH06272210} with the scope of finding the Poincaré series using the Leray--Serre spectral sequence to calculate the cohomologies of the spaces in the following fibration:
	\[
		\begin{tikzcd}[ampersand replacement=\&]
			W^\tau_{(g,n,a)}(r,d) \arrow[r] \& E(G^\tau_{(n,a)}(r)) \times_{G^\tau_{(n,a)}(r)} W^\tau_{(g,n,a)}(r,d) \arrow[r] \& G^\tau_{(n,a)}(r)
		\end{tikzcd}
	\]
	which is the fibration given by the Borel construction for equivariant cohomology.

    Working over $\Zt$ we also have the following isomorphism, that for more general coefficients obviously is not true:
	\[
		\Lambda[a,b] \cong \Lambda[a] \otimes \Lambda[b].
	\]
	Explicitly the isomorphism is given by the linear map that sends $a \mapsto a \otimes 1$ and $b \mapsto 1 \otimes b$.

	This now gives us:
	\[
		\Ho^\star_{G^\tau_{(g+1,0)}}(W^\tau_{(g,g+1,0)}(1,d), \Zt) \cong \Zt[y_{1,1}] \otimes \bigotimes^n_{j=2} \Lambda[y_{j,1}],
	\]
	where $y_{j, 1}$ with $j=1, \ldots, n$ are canonical generators of degree $1$ (see \cite{zbMATH06272210})

	But now we get from Proposition \ref{prop:split}, that the moduli stack splits, and so does its cohomology as we have seen in equation (\ref{eq:kunnethSplit}). First of all we know that
	\[
	\Ho^\star(B\mathbb{R}^\star,\Zt) \cong \Zt[\omega_1],
	\]
	where $\omega_1$ is the first universal Stiefel--Whitney class. Now recall that the cohomology of $\Pic^d(M)$ is canonically generated as an exterior algebra by a basis of $\Ho^1(M,\Zt)$. Also, thanks to the fact that we are considering a type I curve and so $n \neq 0$, we have that $\Pic^d(M)^\sigma$ is given by the invariant divisors, and each connected component is a Lagrangian subvariety of the complex Picard variety, which is still also a real abelian variety. So it is generated by a Lagrangian subset of $\Ho^1(M,\Zt)$ composed by invariant classes (compare \cite{zbMATH03845739}).

	We can therefore conclude that the cohomology algebra of the moduli stack $\Bun^{1,d,\tau}_{M,\sigma}$ where $(M,\sigma)$ is an M--curve, is given as:
	\[
		\Ho^\star(\Bun^{1,d,\tau}_{M,\sigma},\Zt)\cong\Zt[\omega_1] \otimes \Lambda[\beta_i]_{i=1,\dots,n-1}.
	\]

	This gives what we stated in the theorem because for M--curves we have $g^\prime=0$ and also $D$ is null as we have seen in Proposition \ref{prop:dick}.

	Now we do the general case. As we have seen from \cite{zbMATH07144072}, for $(M,\sigma)$ of type I we have that $\sigma$ is equivalent to a reflection through a plane, exchanging to top half with the bottom half, as indicated by the black arrow in the following figure:

	\begin{center}
		\begin{figure}[H]
			\includegraphics[width=0.75\textwidth]{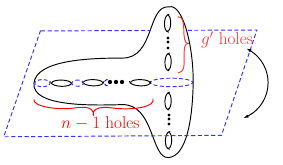}
			\caption{Example of reflection}
			\label{fig:reflection}
		\end{figure}
	\end{center}

	\vspace*{-1.5cm}
	Let us now consider the Eilenberg--Moore spectral sequence (EMss) (see \cite{zbMATH03284042} or \cite{zbMATH07974436} for the general theory) associated to the following pullback diagram (compare \cite{zbMATH06355665})
	\begin{equation}
	\label{eq:BairdDiagram}
		\begin{tikzcd}[ampersand replacement=\&]
			B\tilde{\mathcal{G}}(g^\prime,n,r,\Id) \arrow[r]\arrow[d] \& B\text{Map}_0(M/\sigma \setminus D,U_r) \arrow[d] \\
			\prod^{n}_{i=1}BLU^{\Id}_r \arrow[r] \& \prod^{n}_{i=1}BL_0U_r
		\end{tikzcd}
	\end{equation}
	where $\operatorname{Map_0}(M/\sigma \setminus D,U_r)$ is the subgroup of $\operatorname{Map}(M/\sigma \setminus D,U_r)$ given by those maps that are contractible once restricted to the boundary and where $U_r$ is the unitary group. From \cite[Proposition 6.2]{zbMATH06355665} it follows that the real gauge group is a subgroup of $\tilde{\mathcal{G}}(g^\prime,n,r,\Id)$.

	Applying now the EMss we get the following algebra that lies in the 0--th column of the bigraded algebra $EM^{\star,\star}$ of the $E_2$--term
	\[
		\EM \cong V_{n,1} \otimes A_{g^\prime,1} \otimes \Zt[c_1],
	\]
	where, remembering that $M/\sigma$ is orientable, the Poincaré series of $V_{n,1}$ is equal to $(1+t)^n$ and $A_{g^\prime,1}$ is an exterior algebra with Poincaré series equal to $(1+t)^{2g^\prime}$.
	We obtain this from the fact that $M/\sigma$ is given by a Riemann surface of genus $g^\prime$ with $n$ disks removed. Then taking two copies of it we get back $M$ simply by glueing every boundary circle via the identity map.

	Now following \cite{zbMATH06355665} we obtain a second EMss that, collapsing at the second page, gives us an injection of $\EM$ into the cohomology algebra of the moduli stack:
	\[
		\begin{tikzcd}[ampersand replacement=\&]
			\EM \arrow[r, hook] \& \Ho^\star(\Bun^{1,d,\tau}_{M,\sigma},\Zt) \cong \Ho^\star_{G^\tau_{(n,0)}(1)}(W^\tau_{(g,n,0)}(1,d), \Zt).
		\end{tikzcd}
	\]
	Remembering that we are in the rank one case, we have that this injection is in fact an isomorphism, namely:
	\[
		\EM \cong E(G^\tau_{(n,0)}(1)) \times_{G^\tau_{(n,0)}(1)} W^\tau_{(g,n,0)}(1,d).
	\]

	Therefore we obtain
	\begin{equation}
		\label{eq:firstStepNrank}
		V_{n,1} \otimes \Zt[c_1] \cong \Zt[\omega_1] \otimes \Lambda[\beta_i]_{i=1,\dots,n-1}.
	\end{equation}

	Now we look look at the moduli stack $\Bun^{1,d,\tau}_{M,\sigma}$ of line bundles for a general type I curve $(M, \sigma)$  As before we have an isomorphism between $\EM$ and $\Ho^\star(\Bun^{1,d,\tau}_{M,\sigma},\Zt)$. Also by the construction used in \cite{zbMATH06355665} we note that $V_{n,1} \otimes \Zt[c]$ does not depend on $g^\prime$, and so the isomorphism given in equation (\ref{eq:firstStepNrank}) is still valid. In addition we have another isomorphism (see \cite[Lemma 4.4]{zbMATH06355665})
	\begin{equation}
		\label{eq:secondStepNrank}
		A_{g^\prime,1} \cong \Lambda[\alpha_i]_{i=1,\dots,2g^\prime},
	\end{equation}
	that shows how the classes $\alpha_i, i=1, \ldots, 2g^\prime$ are exactly the images of the generators in the complex case that we have met in equation (\ref{eq:neuGenerator}).

	So we have proved our theorem in the particular case where $(M,\sigma)$ is a type I curve and we get an isomorphism:
	\[
		\Ho^\star(\Bun^{1,d,\tau}_{M,\sigma},\Zt)=\Zt[\omega_1] \otimes \Lambda[\alpha_i]_{i=1,\dots,2g^\prime} \otimes \Lambda[\beta_i]_{i=1,\dots,n-1}.,
	\]

    In order to complete the proof we need to analyze two special cases, namely curves of genus 2 with a single fixed component and curves of genus 3 with two fixed components. These two particular cases are necessary to discuss because they arise for M--curves of genus $0$ and $1$ respectively.

	The two curves that we are analyzing can be pictured as in the following figures, were we have already marked the standard generators:

	\begin{figure}[H]
		\centering
    	\begin{minipage}{0.45\textwidth}
        	\centering
        	\includegraphics[width=\textwidth]{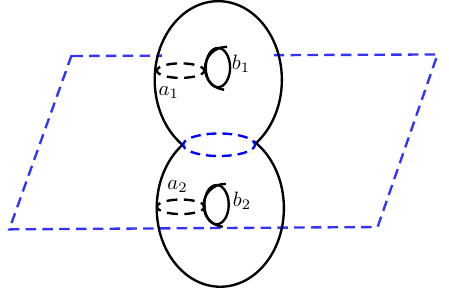}
			\caption{Type $(2,1,0)$}
    	\end{minipage}\hfill
    	\begin{minipage}{0.45\textwidth}
        	\centering
        	\includegraphics[width=\textwidth]{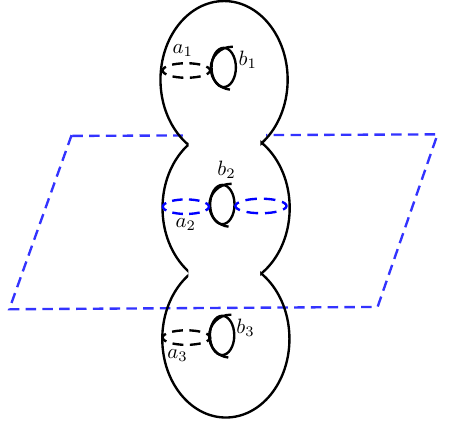}
			\caption{Type $(3,2,0)$}
			\label{fig:320}
    	\end{minipage}
	\end{figure}

	\vspace*{-0.5cm}
	\textit{Case 1:  $(M,\sigma)$ of type $(2,1,0)$}:\\
	We build a basis as we did in Section \ref{sec:basisConstruction}. We have
	\[
		(\alpha_i)_{i=1,\dots,2g^\prime}=(a_1+a_2,b_1+b_2) \text{ and } (A_i)_{i=1,\dots,2g^\prime}=(a_1,b_1).
	\]
	Also we see that in this particular case we have $$\Image(\sigma^\star+\Id)=\Ker(\sigma^\star+\Id).$$ So the classes $(\alpha_i)_{i=1,\dots,2g^\prime}$ must be the ones that generate the cohomology of the moduli stack being the only ones that are $\sigma^\star$--invariant. So we have an isomorphism
	\[
		\Ho^\star(\Bun^{1,d,\tau}_{M,\sigma},\Zt)=\Zt[\omega_1] \otimes \Lambda[\alpha_1,\alpha_2].
	\]

	\textit{Case 2: $(M,\sigma)$ of type $(3,2,0)$}:\\
	Here, as above, we have $g^\prime=1$ and so $D=2$. From this we obtain that
	\[
		(\alpha_i)_{i=1,\dots,2g^\prime}=(a_1+a_3,b_1+b_3) \text{ and } (A_i)_{i=1,\dots,2g^\prime}=(a_1,b_1).
	\]
	We now look at the class $\beta_1$ and the correspondent class $\gamma_1$ and complete the basis.
	Recall that, as in Subsection \ref{sub:basis}, $\beta_1$ is given by the image of a basis of $M/\sigma$ if we consider the reflection as in Figure \ref{fig:320} and so we can argue without loss of generality and get that $\beta_1=a_2$ and $\gamma_1=b_2$. In fact, a basis of $M/\sigma$ is therefore given by the classes $\{a_1+a_3,b_1+b_3,a_2\}$.

	Notice that this set is indeed Lagrangian subset for $\Ho^1(M,\Zt)$. It has dimension $3$ that is half of the dimension of the whole degree $1$ cohomology and if we consider the standard intersection we have that:
	\[
		\alpha_1 \cdot \alpha_2 = (a_1 + a_3) \cdot (b_1 + b_3) = a_1 \cdot b_1 + a_3 \cdot b_3 = 1+1=0,
	\]
	\[
		\alpha_i \cdot \beta_1 = \alpha_i \cdot a_2=0.
	\]

	Summarizing we therefore obtain an isomorphism
	\[
		\Ho^\star(\Bun^{1,d,\tau}_{M,\sigma},\Zt)\cong\Zt[\omega_1] \otimes \Lambda[\alpha_1,\alpha_2] \otimes \Lambda[\beta_1],
	\]
    which completes the proof.
\end{proof}

\section{General rank case for type I curves}

We will now state an prove the main theorem in the general rank case for type I curves and give again a general description of the cohomology algebra of the moduli stack for arbitrary higher rank vector bundles for the case of $M/\sigma$ being orientable.

\begin{theorem}
	\label{th:rkNtheorem}
	Given a type I curve $(M,\sigma)$ of genus $g \geq 2$ and $n$ connected components, the moduli stack $\Bun^{r,d,\tau}_{M,\sigma}$ has the following cohomology algebra:
	\begin{align*}
		\Ho^\star( & \Bun^{r,d,\tau}_{M,\sigma},\Zt)\cong\\
		& \cong\Zt[\omega_k]_{k=1,\dots,r} \otimes \Lambda[\alpha^k_i]_{\substack{i=1,\dots,2g^\prime \\ k=1,\dots,r}} \otimes \Lambda[\beta^k_i]_{\substack{i=1,\dots,n-1 \\ k=1,\dots,r}} \otimes \\
		& \otimes \bigotimes^n_{j=1} \Zt[d^j_{k^\prime}]/((d^j_{k^\prime})^{p_{k^\prime}}) \otimes \Zt[f_k]_{k=2,\dots,r},
	\end{align*}
	where the classes have the following degrees:
	\begin{itemize}
		\item $w_k, \beta^k_i$ have degree $k$,
		\item $\alpha^k_i$ have degree $2k-1$,
		\item $d^j_{k^\prime}$ have degree $k^\prime$,
		\item $f_k$ has degree $2k-2$.
	\end{itemize}

	In addition, $k^\prime$ is an even number such that $k^\prime+1 \leq r$, while $p_{k^\prime}$ is the smallest power of 2 such that $k^\prime \cdot p_{k^\prime} \geq r$.
\end{theorem}

\begin{proof}
	Let us recall the first Eilenberg--Moore Spectral Sequence associated to diagram (\ref{eq:BairdDiagram}) that we have used in the proof of Theorem \ref{th:rk1theorem}. We have the following algebra that lies in the 0--th column of a bigraded algebra:
	\[
		\EM \cong V_{n,r} \otimes A_{g^\prime,r} \otimes \Zt[c_1,\dots,c_r].
	\]
	Now, remembering that $M/\sigma$ is orientable and that we consider now the general rank case, we have that the Poincaré series of $V_{n,r}$ is equal to $\prod^r_{i=1}(1+t^i)^n \prod^{r-1}_{i=1}(1+t^i)^n$ and $A_{g^\prime,r}$ is an exterior algebra with Poincaré series equal to $\prod^r_{i=1}(1+t^{2i-1})^{2g^\prime}$.

	Now we also know that the cohomology of the moduli stack is given via the fibration of equation (\ref{eq:LSfiber}) (compare \cite{zbMATH06272210}). The Leray-Serre spectral sequence (LSss) associated to this fibration collapses because $\tau$ is a real structure and type I curves have $n > 0$. So we have that:
	\[
		\Ho^\star(\Bun^{r,d,\tau}_{M,\sigma},\Zt) \cong \Ho^\star(B(\Omega^2(U(r))),\Zt) \otimes \Ho^\star_{G^\tau_{(n,a)}(r)}(W^\tau_{(g,n,a)}(r,d),\Zt).
	\]

	As we already saw from \cite{zbMATH06355665} we also get a second EMss which collapses at the second page. From this we get again an inclusion of $\EM$ into the whole cohomology algebra of the moduli stack.

	We will proceed with the proof in several steps, starting with what we know in the rank one case and then going up with the rank.

	\textit{First step:}\\
	Take the Klein surface $(M,\sigma)$ with $g=n-1$ and consider now the general rank case.

	From the work of Liu--Schaffhauser \cite{zbMATH06272210} and from Theorem \ref{th:rk1theorem} we have that:
	\[
		\Ho^\star_{G^\tau_{(n,0)}(r)}(W^\tau_{(n-1,n,0)}(r,d),\Zt) \cong \Zt[\omega_k]_{k=1,\dots,r} \otimes \Lambda[\beta^k_i]_{\substack{i=1,\dots,n-1 \\ k=1,\dots,r}} \otimes S,
	\]
	where the algebra $S$ has a Poincaré series equal to $\prod^{n}_{i=1}P_t(BO(r),\Zt)$.

	On the other hand we can follow the alternative construction of $\EM$ following \cite{zbMATH06355665}. First we consider the cohomology algebra $\Ho^\star\left(\prod^n_{i=1}BLU^{\Id}_r,\Zt\right)$. It is isomorphic to the following tensor product by simply using the Künneth formula and definition \ref{def:loop}:
	\[
		\Ho^\star\left(\prod^n_{i=1}BLU^{\Id}_r,\Zt\right) \cong \bigotimes^n_{i=1} \Ho^\star(SO(r),\Zt) \otimes \Zt[\omega^k_i]_{\substack{i=1,\dots,n \\ k=1,\dots,r}}.
	\]
	To determine $\EM$ we need to tensorize it with the Koszul complex and then apply the Koszul differential (the interior one, see \cite[Lemma 6.4]{zbMATH06355665}) to it. For this we use the following algebra isomorphism:
	\[
		\Zt[a] \cong \Lambda[a] \otimes \Zt[a^2].
	\]
	This will give us the following description as a graded algebra of the second page of our EMss:
	\begin{equation}
		\label{eq:fourthStepNrank}
		\EM \cong \Zt[\omega_k]_{k=1,\dots,r} \otimes \Lambda[\beta^k_i]_{\substack{i=1,\dots,n-1 \\ k=1,\dots,r}} \otimes \bigotimes^n_{i=1} \Ho^\star(SO(r),\Zt).
	\end{equation}
	The classes $\beta^k_i$ are given by the restriction of the universal Stiefel--Whitney classes over the $k$--th connected component of $M^\sigma$ as given by the cohomology of the real loop group.

	We will later see that, as expected, we have the following isomorphism:
	\[
		S \cong \bigotimes^n_{i=1} \Ho^\star(SO(r),\Zt).
	\]

	We need to describe the generators of the components. Exploiting \cite[Remark 6.6]{zbMATH06355665} we have the following injection:
	\[
		\begin{tikzcd}[ampersand replacement=\&]
			\Ho^\star(SO(r),\Zt) \arrow[r, hook] \& \EM.
		\end{tikzcd}
	\]

	We continue now by describing the generators of $\Ho^\star(SO(r),\Zt)$.
	First of all, we have a homotopy equivalence between $SO(r)$ and $\Omega BSO(r)$, so calculating the cohomology of $\Omega BSO(r)$ will give us the cohomology of $SO(r)$.

	We have that (compare \cite[Example 5.H]{zbMATH01565334})
	\[
		\Ho^\star(BSO(r),\Zt) \cong \Zt[w_2, \dots,w_r].
	\]
	From this we obtain the cohomology algebra of $\Omega BSO(r)$ by using the following Lemma which is a consequence of \cite[Theo. 1]{zbMATH06791174}.

	\begin{lemma}
		\label{lm:samLemma}
		The mod $2$ cohomology of $\Omega BSO(r)$ is given by the following algebra:
		\[
		\Ho^\star(\Omega BSO(r), \Zt)=\Zt[\overline{w_{2k}}]/I,
		\]
		where
		\begin{itemize}
			\item $\overline{w_{2k}}$ are the suspension classes of the even degree universal Stiefel--Whitney classes,
			\item $I$ is the ideal generated by $\overline{w_{2k}}^{p_{2k}}$, with $p_{2k}$ being the smallest power of $2$ such that $(2k-1) \cdot p_{2k} \geq r$.
		\end{itemize}
	\end{lemma}

	We will give the proof of this lemma after finishing the proof of the main theorem.

	\textit{Second step:}\\
	Now we have to take a generic Klein surface and consider the generic rank case using the previous considerations.

	First of all we notice that the only part that is now changing in $\EM$ is the one given by $A_{g^\prime,r}$. By \cite[Lemma 6.4]{zbMATH06355665} we know that it is isomorphic to an exterior algebra in $2g^\prime \cdot r$ generators of odd degrees.

	Following the proof of \cite[Lemma 6.4]{zbMATH06355665} we notice that the generators are the same of the classical generators for the cohomology of the moduli stack of complex vector bundles (see \cite{zbMATH05605320}, and equation (\ref{eq:neuGenerator}), namely the ones that are denoted there as $a^{(i)}_j$).

	We have now to choose which ones are the right $2g^\prime \cdot r$ generators. But it is clear that the correct ones are, after fixing the basis that we have constructed in Section \ref{sec:basisConstruction}, the generators given by the images of the generators of the map $\sigma^\star+\Id$. So we get:
	\begin{equation}
		\label{eq:fifthStepNrank}
		A_{g^\prime,r} \cong \Lambda[\alpha^k_i]_{\substack{i=1,\dots,2g^\prime \\ k=1,\dots,r}},
	\end{equation}
	where the degree of $\alpha^k_i$ is equal to $2k-1$.

	\textit{Third step:}\\
	Now we want to prove that also in the general case the following isomorphism still holds:
	\begin{equation}
		\label{eq:thirdStepNrank}
		\EM \cong \Ho^\star_{G^\tau_{(n,0)}(r)}(W^\tau_{(g,n,0)}(r,d),\Zt).
	\end{equation}

	First we notice that the two algebras have the same Poincaré series once we fix all the parameters.

	These two algebras both live inside the cohomology algebra of the moduli stack, and they depend on the same parameters in the same way.

	Consider now the following maps:
	\[
		\begin{tikzcd}[ampersand replacement=\&]
			\EM \arrow[r, hook] \& \Ho^\star(\Bun^{r,d,\tau}_{M,\sigma}, \Zt) \arrow[r, two heads] \& \Ho^\star_{G^\tau_{(n,0)}(r)}(W^\tau_{(g,n,0)}(r,d),\Zt),
		\end{tikzcd}
	\]
	where the first map is given by the injection while the second one is given by the projection to the second factor through the isomorphism:
	\[
		\Ho^\star(\Bun^{r,d,\tau}_{M,\sigma},\Zt) \cong	\Ho^\star(B(\Omega^2(U(r))),\Zt) \otimes \Ho^\star_{G^\tau_{(n,a)}(r)}(W^\tau_{(g,n,a)}(r,d),\Zt).
	\]

	In the particular case that the image of the injection lies completely in the second factor we are finished, so we want to show that the following map
	\begin{equation}
		\label{eq:zeroMap}
		\begin{tikzcd}[ampersand replacement=\&]
			\EM \arrow[r, hook] \& \Ho^\star(\Bun^{r,d,\tau}_{M,\sigma}, \Zt) \arrow[r, two heads] \& \Ho^\star(B(\Omega^2(U(r))),\Zt),
		\end{tikzcd}
	\end{equation}
	is constant zero.

	We proceed by induction on the rank. The base case, i. e. rank equal to one, was already analyzed in theorem \ref{th:rk1theorem}. It is true, because $\Ho^\star(B(\Omega^2(U(r))),\Zt)$ is trivial when $r=1$. Now let's do the inductive step.

	When we consider for the rank $r \mapsto r+1$ in $\EM$ we get some new generators, namely the following:
	\begin{enumerate}
		\item One degree $r+1$ generator of a polynomial algebra,
		\item $n-1$ degree $r+1$ generators of an exterior algebra,
		\item $2g^\prime$ degree $2(r+1)-1$ generators of an exterior algebra,
		\item Eventually $n$ degree $r$ generators given by the change from $\Ho^\star(SO(r),\Zt)$ to $\Ho^\star(SO(r+1),\Zt)$.
	\end{enumerate}

	In addition, we also know that the cohomology of $B(\Omega^2(U(r+1)))$ gives a new generator of degree $2r$ and that it is in fact a polynomial algebra.

	Now we notice that looking at the degrees of the classes, the only one that, by a degree argument, can survive are the ones from (4) once we square them.

	But, by construction of the cohomology of $\Ho^\star(SO(r+1),\Zt)$, we have that a rank $r$ generator, if it exists, is sent to $0$ by the square map.

	This proves that the map described in equation (\ref{eq:zeroMap}) is always equal to the constant zero map and so we get the desired isomorphism of graded algebras.

	\textit{Fourth step:}\\
	We can now finish the proof for any type I curve of general type $(g,n,0)$. We  have that:
	\begin{align*}
		\EM \cong & \Zt[\omega_k]_{k=1,\dots,r} \otimes \Lambda[\alpha^k_i]_{\substack{i=1,\dots,2g^\prime \\ k=1,\dots,r}} \otimes \Lambda[\beta^k_i]_{\substack{i=1,\dots,n-1 \\ k=1,\dots,r}} \otimes \\
		& \otimes \bigotimes^n_{j=1} \Zt[d^j_{k^\prime}]/((d^j_{k^\prime})^{p_{k^\prime}}),
	\end{align*}
	but this algebra is also isomorphic to
	\[
		\Ho^\star_{G^\tau_{(n,0)}(r)}(W^\tau_{(g,n,0)}(r,d),\Zt)
	\]
	thanks to equation (\ref{eq:thirdStepNrank}).

	Finally using the diagram \cite[Sec. 4.3, eq (4.2)]{zbMATH06272210}, we get that the last generators needed are exactly the same as those of the cohomology algebra of the moduli stack in the complex case reduced modulo $2$. This gives us finally the desired isomorphism of graded algebras
	\begin{align*}
		\Ho^\star( & \Bun^{r,d,\tau}_{M,\sigma},\Zt)\cong \\
		& \cong\Zt[\omega_k]_{k=1,\dots,r} \otimes \Lambda[\alpha^k_i]_{\substack{i=1,\dots,2g^\prime \\ k=1,\dots,r}} \otimes \Lambda[\beta^k_i]_{\substack{i=1,\dots,n-1 \\ k=1,\dots,r}} \otimes \\
		& \otimes \bigotimes^n_{j=1} \Zt[d^j_{k^\prime}]/((d^j_{k^\prime})^{p_{k^\prime}}) \otimes \Zt[f_k]_{k=2,\dots,r},
	\end{align*}
	which concludes the proof of the theorem.
\end{proof}

\begin{proof}[Proof of Lemma \ref{lm:samLemma}]
	In order to apply \cite[Theorem 1]{zbMATH06791174}, we need to have a simply connected space that has a polynomial algebra as cohomology. But the space under consideration here is $BSO(r)$ so it satisfies the conditions.

	By the theorem it exists a subset $\mathcal{S}$ of the set of generators $\{w_2,\dots,w_r\}$. This subset is defined taking the generators $w_k$ such that:
	\[
		w_k \notin \Image Sq_1 \; \mod \HHplus,
	\]
	where $Sq_1$ is the mod $2$ Steenrod operator
	\[
		\begin{tikzcd}[ampersand replacement=\&]
			Sq_1: H^n(X,\Zt) \arrow[r] \& H^{2n-1}(X,\Zt),
		\end{tikzcd}
	\]
	defined as $Sq_1(w_k)=Sq^{k-1}(w_k)$, while $\HHplus$ is the subalgebra of the cohomology generated by the product of two classes with positive degree.

	Now for applying \cite[Theorem 1]{zbMATH06791174} we need to check the $\smile_1$--height of the generators $w_k \in \mathcal{S}$. By definition this is the smallest integer $\nu_k$ such that applying $\nu_k +1$ times the Steenrod operator $Sq_1$ to $w_k$ leads to an element that is in $\HHplus$. Doing this we will also see that $w_k \in \mathcal{S}$ if and only if $k$ is even.

	First we analyze the rank cases $r=2$ and $r=3$ in order to understand the general case.

	\textit{Case rank $r=2$:}\\
	For rank $2$ we have to consider only $w_2$.  Doing the calculation using the Wu formula (see \cite[Section 8, Problem 8-A]{zbMATH03468033}), we get
	\[
		Sq_1(w_2)=Sq^1(w_2)=w_1w_2,
	\]
	because $w_3=0$ in this case. So we have that $w_2 \in \mathcal{S}$ and $\nu_2=0$.

	\textit{Case rank $r=3$:}\\`'
	Now we have the following relation by applying $Sq_1$ to the class $w_2$:
	\[
		Sq_1(w_2)=Sq^1(w_2)=w_1w_2+w_3,
	\]
	so we see directly that $w_3 \notin \mathcal{S}$. Also applying $Sq^1$ another time, but now to $w_3$ and because $Sq_1$ sends $\HHplus$ into $\HHplus$, we get:
	\[
		Sq_1(w_3)=Sq^2(w_3)=w_2w_3.
	\]
	From this we therefore get $\nu_2=1$. And this is correct, because we want the smallest integer such that: $1 \cdot 2^{\nu_2+1} \geq 3$, and it is exactly $\nu_2=1$.

	Now we can generalize this procedure. Take a class $w_k$ with $k=2k^\prime$ even. We now get that, considering that some classes can be zero:
	\[
		Sq_1(w_k)=w_{4k^\prime-1}+\sum^{k-2}_{i=0}w_{k-i-1}w_{k+i}.
	\]
	So for each $r < 4k^\prime-1$, we have $\nu_k=0$ and in any case $w_{4k^\prime-1} \notin \mathcal{S}$ for each value of $k^\prime$. Also it is easy to check that:
	\[
		(k-1) \cdot 2^{\nu_k+1} = (k-1) \cdot 2 = 2k-2 = 4k^\prime-2 \geq r.
	\]

	If $r \geq 4k^\prime-1$, we get $\nu_k \neq 0$ and we have to apply again $Sq_1$ to the $(4k^\prime-1)$--degree classes. We use again the Wu formula and we obtain:
	\[
		Sq_1(w_{4k^\prime-1})=w_{8k^\prime-3}+\sum^{4k^\prime-3}_{i=0}w_{4k^\prime-i-2}w_{4k^\prime+i-1}.
	\]
	Again we get odd degree classes not in $\mathcal{S}$ and we can do so again by similar reasoning as above. Generalizing this we therefore get that for each choice of integers $i \geq 2,j \geq 1$ we have that the odd class of degree
	\[
		2^ij-2^{i-1}+1=2^{i-1}(2j-1)+1,
	\]
	does not live in $\mathcal{S}$.

	We want to prove that there exists a choice of $i \geq 2,j \geq 1$ that gives us every odd number, i. e.  that for each integer $k \geq 1$ there exists a choice of $i \geq 2,j \geq 1$ such that:
	\[
		2^{i-2}(2j-1)=k.
	\]
	But this can be simply done exploiting the unique factorization of $k$. In particular if $k$ can be factorized as $2^l \cdot k^\prime$, we have that $k^\prime$ is odd and we can take the pair $(l+2,\frac{k^\prime+1}{2})$ as the pair $(i,j)$.

	Now, as we already observed in the case of rank $3$, if we take the smallest power of $2$ such that $(2k^\prime-1) \cdot p_{2k^\prime} \geq r$ we have the following chain of inequalities:
	\[
		r \leq (2k^\prime-1) \cdot p_{2k^\prime}=(2k^\prime-1) \cdot 2^{\nu_{2k^\prime}+1}=2^{\nu_{2k^\prime}+2}k^\prime-2^{\nu_{2k^\prime}+1} < 2^{\nu_{2k^\prime}+2}k^\prime-2^{\nu_{2k^\prime}+1}+1,
	\]
	where the last number is exactly the degree of the class that should be paired with $w_0$ in the Wu formula once we apply $\nu_{2k^\prime}+1$ times $Sq_1$ to the class $w_{2k^\prime}$.

	Now we have established all the conditions of \cite[Theorem 1]{zbMATH06791174} and we have got all the parameters needed for applying the theorem and therefore have finished the proof of Lemma \ref{lm:samLemma}.
\end{proof}

\section*{Acknowledgments}
\noindent The first author is truly grateful to his supervisor Roberto Pignatelli for the valuable discussions and insights he provided. The second author is partially supported by PRIN project Moduli spaces and special varieties (2022). He also acknowledges the University of Pavia for financial support through a ContRic-DSM56 grant. Both authors are members of INdAM-GNSAGA. The figures in this paper were created using the software package LaTeXDraw, which we learned about from Daniel Dugger. The authors are grateful to him for sharing the code used in the illustrations of \cite{zbMATH07144072}. The authors also gratefully acknowledge Florent Schaffhauser for pointing out several errors in an earlier draft of this article.

\printbibliography

\end{document}

\typeout{get arXiv to do 4 passes: Label(s) may have changed. Rerun}